\documentclass[12pt,a4paper,oneside]{amsart}

\usepackage[left=18mm,right=18mm,top=23mm,bottom=13mm,footskip=6.2mm]{geometry}
\usepackage[sc]{mathpazo}
\usepackage[T1]{fontenc}
\usepackage{amsmath,amssymb,mathtools}
\usepackage{parskip}
\usepackage{etoolbox}
\usepackage{microtype}
\usepackage{xcolor}
\usepackage[normalem]{ulem}

\definecolor{cancelred}{HTML}{D9534F}
\definecolor{matchblue}{HTML}{0275D8}
\definecolor{matchgreen}{HTML}{5CB85C}
\definecolor{highlightpurple}{HTML}{6F42C1}
\definecolor{CheckGreen}{RGB}{34,139,34}
\definecolor{WrightYellow}{RGB}{218,165,32}
\definecolor{linkviolet}{RGB}{112,60,160}
\definecolor{citeblue}{RGB}{35,82,145}

\usepackage{hyperref}
\hypersetup{
  colorlinks=true,
  linkcolor=linkviolet,
  citecolor=citeblue,
  urlcolor=citeblue,
  pdftitle={Isomorphic one-relator groups need not have relators of equal stable commutator length},
  pdfauthor={Artem Semidetnov},
  pdfsubject={One-relator groups and stable commutator length}
}

\patchcmd{\section}{\scshape}{\large\bfseries}{}{}
\patchcmd{\abstract}{\leftmargin3pc}{\leftmargin2pc}{}{}
\makeatletter
\renewcommand{\@secnumfont}{\bfseries}
\makeatother
\patchcmd{\thebibliography}{\section*}{\paragraph}{}{}

\DeclareMathOperator{\scl}{scl}
\newcommand{\nc}[1]{\langle\!\langle #1 \rangle\!\rangle}
\newcommand{\genrel}[2]{\langle #1 \mid #2 \rangle}
\newcommand{\blue}[1]{\textcolor{matchblue}{#1}}

\theoremstyle{plain}
\newtheorem{theoremA}{Theorem}

\newtheorem{theorem}{Theorem}
\newtheorem{lemma}[theorem]{Lemma}
\newtheorem{corollary}[theorem]{Corollary}
\theoremstyle{definition}

\newtheorem*{Question}{Question}

\title[Stable commutator length of a relator]{The stable commutator length of a relator\\ is not a one-relator group invariant}
\author{Artem Semidetnov}
\address{Universit\'{e} de Gen\`{e}ve, Section de math\'{e}matiques, rue du Conseil-G\'{e}n\'{e}ral 7--9, 1205 Gen\`{e}ve, Switzerland}
\email{artemsemidetnov@gmail.com, Artem.Semidetnov@etu.unige.ch}
\date{}

\begin{document}

\vspace*{-15mm}

\begin{abstract}
We show that, for the words~\mbox{{\scriptsize \(r=\mathtt{aabABabABBAbaabABBAb}\)}} and~\mbox{{\scriptsize \(r'=\mathtt{aabABabABabABBAbaBAb}\)}}, both in the commutator subgroup of the free group on~\(\mathtt{a}\) and~\(\mathtt{b}\), the one-relator groups~\mbox{\(\langle \mathtt{a,b}\mid r=1\rangle\)} and~\mbox{\(\langle \mathtt{a,b}\mid r'=1\rangle\)} are isomorphic, whereas~\(r\) and~\(r'\) have distinct stable commutator lengths. This provides a negative answer to a question posed by Heuer and L\"oh.
\end{abstract}

\maketitle

Let \(S\) be a set and let \(F(S)\) be the free group on \(S\). For \(w\in F(S)'\), write \(\operatorname{cl}_S(w)\) for its commutator length and set
\[
  \scl_S(w):=\lim_{n\to\infty}\frac{\operatorname{cl}_S(w^n)}{n}.
\]
Duncan and Howie proved the gap theorem [3, Theorem~3.3]:
\[
  w\in F(S)'\setminus\{e\}\quad\Longrightarrow\quad \scl_S(w)\geq \frac12.
\]
Moreover, stable commutator length in free groups is rational and computable [2]. Heuer and L\"oh asked whether this quantity can be recovered from the corresponding one-relator group.

\begin{Question}[{\normalfont [4, Question~1.3(1)]}]
Let \(S,S'\) be sets and let \(r\in F(S)'\setminus\{e\}\), \(r'\in F(S')'\setminus\{e\}\) be relators with \(\genrel{S}{r}\cong\genrel{S'}{r'}\). Does this imply that \(\scl_S r=\scl_{S'}r'\)?
\end{Question}

For a non-trivial \(r\in F(S)'\), put
\[
  G_r:=\genrel{S}{r}=F(S)/\nc r.
\]
Following [4], let \(\lVert G_r\rVert\) denote the \(\ell^1\)-seminorm of the fundamental class of \(G_r\), called the simplicial volume of \(G_r\). Heuer and L\"oh proved the universal estimate [4, Corollary~3.12]:
\[
  \lVert G_r\rVert<4\scl_S(r).
\]
They also posed the following question.

\begin{Question}[{\normalfont [4, Question~1.2]}]
For which non-trivial \(r\in F(S)'\) does one have
\[
  \lVert G_r\rVert=4\left(\scl_S(r)-\frac12\right)?
\]
\end{Question}

We give a single example showing that the stable commutator length of a relator is not determined by its one-relator group and, at the same time, that the equality in Question~1.2 need not hold.

\begin{theoremA}\label{thm:main}
Let \(S=S'=\{\mathtt a,\mathtt b\}\) and
\[
  r=\mathtt{aabABabABBAbaabABBAb},\qquad
  r'=\mathtt{aabABabABabABBAbaBAb}.
\]
The words \(r,r'\in F(S)'\setminus\{e\}\) are cyclically reduced and have length~\(20\). Moreover,
\[
  G_r\cong G_{r'},\qquad \scl_S(r)=1,
  \qquad\text{and}\qquad \scl_S(r')=\frac12.
\]
In particular, the answer to [4, Question~1.3(1)] is negative.
\end{theoremA}

\begin{corollary}\label{cor:simvol}
For \(r=\mathtt{aabABabABBAbaabABBAb}\),
\[
  \lVert G_r\rVert<2
  =4\left(\scl_S(r)-\frac12\right).
\]
Thus the equality in [4, Question~1.2] does not hold for every non-trivial relator in \(F(S)'\).
\end{corollary}

We now record a convenient criterion for checking the isomorphism in Theorem~\ref{thm:main}.

\begin{lemma}\label{lem:cert}
Let \(r\in F(\mathtt a,\mathtt b)\), \(r'\in F(\mathtt x,\mathtt y)\), let \(p,q\in F(\mathtt a,\mathtt b)\), and let \(s,t\in F(\mathtt x,\mathtt y)\). Define homomorphisms
\[
\begin{aligned}
  \varphi\colon F(\mathtt x,\mathtt y)&\longrightarrow F(\mathtt a,\mathtt b),
  &\varphi(\mathtt x)&=p,&\varphi(\mathtt y)&=q,\\
  \psi\colon F(\mathtt a,\mathtt b)&\longrightarrow F(\mathtt x,\mathtt y),
  &\psi(\mathtt a)&=s,&\psi(\mathtt b)&=t.
\end{aligned}
\]
Assume that
\begin{enumerate}
  \item \(\varphi(r')\in\nc r\) and \(\psi(r)\in\nc{r'}\);
  \item \(\psi(p)\mathtt x^{-1},\psi(q)\mathtt y^{-1}\in\nc{r'}\);
  \item \(\varphi(s)\mathtt a^{-1},\varphi(t)\mathtt b^{-1}\in\nc r\).
\end{enumerate}
Then \(\genrel{\mathtt{x,y}}{r'}\cong\genrel{\mathtt{a,b}}r\).
\end{lemma}

\begin{proof}
The first pair of conditions lets \(\varphi\) and \(\psi\) descend to the two quotient groups. The second and third pairs say that the induced maps fix the respective generators after composition, so they are mutually inverse.
\end{proof}

\begin{proof}[Proof of Theorem~\ref{thm:main}]\small
The two displayed words are visibly cyclically reduced and of length~\(20\). Their exponent sums in both generators vanish, so they lie in \(F(S)'\).

\emph{Stable commutator length.}
The identity
\[
  r'=\mathtt a\,[\mathtt{ab},\mathtt{ABabABabA}]\,\mathtt A
\]
shows that \(r'\) is conjugate to a commutator. Hence \(\scl_S(r')\leq\frac12\), and the Duncan--Howie gap gives equality. An exact computation with Walker's \texttt{scallop}, implementing Calegari's algorithm, gives \(\scl_S(r)=1\); the reproducible computation is included in [5].

\emph{The isomorphism.}
Relabel \(r'\) by \(\mathtt a\mapsto\mathtt x\) and \(\mathtt b\mapsto\mathtt y\), so that
\[
  r'=\mathtt{xxyXYxyXYxyXYYXyxYXy}.
\]
In Lemma~\ref{lem:cert}, take
\[
  p=\mathtt{BAbaa},\quad q=\mathtt{AABabba},\quad
  s=\mathtt{YxyXYxy},\quad t=\mathtt{yxYXy}.
\]
We verify its three sets of hypotheses. Throughout, \(\sim\) denotes cyclic conjugacy, and the initial segment moved in each cyclic shift is shown in \blue{blue}.
Every cyclic conjugate of \(r^{\pm1}\) lies in \(\nc r\), and~likewise~for~\(r'\); thus each product displayed below is an explicit normal-closure certificate.

First, four of the required certificates are single cyclic conjugates:
\begin{align*}
  \varphi(s)\mathtt a^{-1}
    &=\mathtt{ABBAb}\blue{\mathtt{aabABabABBAbaab}}\sim r,
  &\varphi(t)\mathtt b^{-1}
    &=\mathtt{AABabbaBAbaBAA}\blue{\mathtt{BabbaB}}\sim r^{-1},\\
  \psi(p)\mathtt x^{-1}
    &=\mathtt{YxyXYYXyxYXy}\blue{\mathtt{xxyXYxyX}}\sim r',
  &\psi(q)\mathtt y^{-1}
    &=\mathtt{YXyxYXX}\blue{\mathtt{YxyXYxyyxYXyx}}\sim (r')^{-1}.
\end{align*}
Second, \(\varphi(r')\) is a product of three cyclic conjugates of \(r^{\pm1}\):
\[
\begin{aligned}
\varphi(r')={}&
 \underbrace{\mathtt{BAb}\blue{\mathtt{aabABabABBAbaabAB}}}_{\sim r}\,
 \underbrace{\mathtt{abABBAb}\blue{\mathtt{aabABabABBAba}}}_{\sim r}\,
 \underbrace{\mathtt{BAABabbaBAbaBAA}\blue{\mathtt{Babba}}}_{\sim r^{-1}}.
\end{aligned}
\]
Finally, set
\begin{align*}
 w_1&=\mathtt{YxyXYxxyXYxyyxYXXYxyXYxyyxYX},
 &c_1&=\mathtt{XYxyXYYXyxYXyxxyXYxy}\sim r',\\
 w_2&=\mathtt{YxyXYxxyXYxy},
 &c_2&=\mathtt{yxYXXYxyXYxyyxYXyxYX}\sim (r')^{-1},\\
 w_3&=\mathtt{YxyXY},
 &c_3&=\mathtt{xxyXYxyXYxyXYYXyxYXy}=r'.
\end{align*}
Substituting the displayed words and cancelling adjacent inverse pairs gives
\[
  \psi(r)=w_1c_1w_1^{-1}\,w_2c_2w_2^{-1}\,w_3c_3w_3^{-1},
\]
and therefore \(\psi(r)\in\nc{r'}\). Lemma~\ref{lem:cert} now yields \(G_r\cong G_{r'}\).
\end{proof}

Because stable commutator length is invariant under automorphisms of the free group, conjugation, and inversion, Theorem~\ref{thm:main} also shows that no automorphism of \(F(S)\) sends \(r\) to a conjugate of \((r')^{\pm1}\). Thus the isomorphism above only appears after passing to the one-relator quotients.

\begin{proof}[Proof of Corollary~\ref{cor:simvol}]\small
Simplicial volume is an isomorphism invariant, so Theorem~\ref{thm:main} gives \(\lVert G_r\rVert=\lVert G_{r'}\rVert\). By [4, Corollary~3.12],
\[
  \lVert G_r\rVert=\lVert G_{r'}\rVert
  <4\scl_S(r')=2
  =4\left(\scl_S(r)-\frac12\right).\qedhere
\]
\end{proof}

\section*{Computational Methodology}
The search began with cyclically reduced relators in \(F_2'\) whose one-relator groups are isomorphic but which do not lie in the same \(\operatorname{Aut}(F_2)\)-orbit. Claude Desktop with Opus~5 designed and orchestrated an exhaustive search through \(\operatorname{Aut}(F_2)\)-orbits of words of length at most~\(20\). Candidate pairs were filtered using the first homology of low-index subgroups, Alexander polynomials, and counts of homomorphisms to small finite groups. For candidates not separated by these invariants, we constructed explicit isomorphisms and compared stable commutator lengths with \texttt{scallop} [6].

The code, search pipeline, and explicit certificates are available in [5]. The same repository reproduces the computation \(\scl_S(r)=1\) and contains a Lean formalisation of the six identities used above, produced with the assistance of Harmonic's Aristotle [1].

\smallskip
\noindent\textbf{AI use statement.} Claude Desktop with Opus~5 assisted in designing and orchestrating the computational search for candidate counterexamples.

\smallskip
\noindent\textbf{Acknowledgement.} The author thanks Clara Löh and Nicolaus Heuer for their attention to this solution and for their helpful correspondence. The author is grateful to Vasily Ionin for stimulating discussions and careful reading of a draft of this paper.

\vskip 0.4em
{\footnotesize\setlength{\parindent}{0pt}\setlength{\parskip}{0pt}\setlength{\hangindent}{1.4em}
\textbf{References}\par
\hangindent=1.4em [1] T.~Achim et al., \uline{Aristotle: IMO-level automated theorem proving}, preprint (2025), \href{https://arxiv.org/abs/2510.01346}{arXiv:2510.01346}.\par
\hangindent=1.4em [2] D.~Calegari, \href{https://doi.org/10.1090/S0894-0347-09-00634-1}{\uline{Stable commutator length is rational in free groups}}, J.~Amer. Math. Soc. \textbf{22} (2009), no.~4, 941--961.\par
\hangindent=1.4em [3] A.~J.~Duncan and J.~Howie, \href{https://doi.org/10.1007/BF02571522}{\uline{The genus problem for one-relator products of locally indicable groups}}, Math.~Z. \textbf{208} (1991), no.~1, 225--237.\par
\hangindent=1.4em [4] N.~Heuer and C.~L\"oh, \href{https://doi.org/10.2140/agt.2022.22.1615}{\uline{Simplicial volume of one-relator groups and stable commutator length}}, Algebr. Geom. Topol. \textbf{22} (2022), no.~4, 1615--1661.\par
\hangindent=1.4em [5] A.~Semidetnov, \uline{Code, search pipelines and certificates accompanying this paper}, \url{https://github.com/s3midetnov/scl-one-relator}.\par
\hangindent=1.4em [6] A.~Walker, \uline{\texttt{scallop}}, \url{https://github.com/aldenwalker/scallop}.\par}

\end{document}